\documentclass[11pt]{amsart}
\usepackage{amsthm}
\usepackage{amssymb}
\usepackage{mathrsfs}
\usepackage{tikz}
\usepackage[margin=1in]{geometry}
\usepackage{color}
\usepackage{enumitem}
\usepackage{bbm}
\usepackage{bm}
\usepackage{verbatim}
\usepackage{todonotes}
\usepackage[
	colorlinks,
	linkcolor={black},
	citecolor={black},
	urlcolor={black}
]{hyperref}
\usepackage{cite}

\newcommand{\bbD}{{\mathbbm{D}}}

\newcommand{\bbP}{{\mathbbm{P}}}
\newcommand{\bbQ}{{\mathbbm{Q}}}
\newcommand{\T}{{\mathbbm{T}}}
\newcommand{\bbR}{{\mathbbm{R}}}

\newcommand{\bbZ}{{\mathbbm{Z}}}

\newcommand{\SL}{{\mathrm{SL}}}

\newcommand{\tr}{{\mathrm{Tr}}}

\renewcommand{\Im}{\operatorname{Im}}
\renewcommand{\Re}{\operatorname{Re}}

  \DeclareMathOperator\supp{supp}

\allowdisplaybreaks

\newtheorem{theorem}{Theorem}[section]
\newtheorem{prop}[theorem]{Proposition}
\newtheorem{coro}[theorem]{Corollary}
\newtheorem{lemma}[theorem]{Lemma}

\theoremstyle{definition}
\newtheorem{definition}[theorem]{Definition}
\newtheorem{remark}[theorem]{Remark}

\definecolor{purple}{rgb}{.5,0,1}
\definecolor{orange}{rgb}{1,.5,0}
\definecolor{green}{rgb}{0,.4,0}

\numberwithin{equation}{section}

\title[Opening Gaps in the Spectrum of Strictly Ergodic Jacobi and CMV Matrices]{Opening Gaps in the Spectrum of Strictly Ergodic\\Jacobi and CMV Matrices}

\author{David Damanik}
\address{Department of Mathematics, Rice University, Houston, Texas, 77005}
\email{damanik@rice.edu}
\thanks{D.D.\ was supported in part by NSF grant DMS--2054752}

\author{Long Li}
\address{Department of Mathematics, Rice University, Houston, Texas, 77005}
\email{longli@rice.edu}

\begin{document}
\maketitle

\begin{abstract}
We prove that dynamically defined Jacobi and CMV matrices associated with generic continuous sampling functions have all gaps predicted by the Gap Labelling Theorem open. We also give a mechanism for generic gap opening for quasi-periodic analytic sampling functions in the subcritical region following from the analyticity of resonance tongue boundaries for both Jacobi and CMV matrices.
\end{abstract}

\setcounter{tocdepth}{1}

\tableofcontents

\section{Introduction} 

Given a dynamically defined family $\{H_\omega\}_{\omega \in \Omega}$ of self-adjoint (or unitary) operators and an ergodic probability measure $\mu$, the spectrum turns out to be almost surely equal to the same set $\Sigma$. A result of this nature is typically shown for certain classes of operators, defined on suitable Hilbert spaces. Specifically, the standard textbook treatments \cite{CL90, CFKS, DF22, PF92} focus on the case of Schr\"odinger operators on $\ell^2(\bbZ^d)$ or $L^2(\bbR^d)$, but the same proof works in more general settings.

Fixing an ergodic measure $\mu$, it is of course of interest to study the structure of the set $\Sigma$. It is convenient to view $\Sigma$ as obtained by removing its interior gaps from its convex hull:
\begin{equation}\label{e.gaps}
\Sigma = \mathrm{ConvHull}(\Sigma) \setminus \bigcup_{j \in J} (g_j^{(-)}, g_j^{(+)}).
\end{equation}

It is a crucial observation that the underlying dynamics may impose restrictions on the structure of these interior gaps. In order to exhibit this phenomenon, one first observes that there is a natural labelling of the gaps and then proceeds to show that these labels are confined to a suitable (at most) countable set $\mathcal{L}$ that is completely determined by the dynamics and must be the same for all operator families that can be generated by the dynamics.

For definiteness let us discuss the case where the operators in question are bounded self-adjoint operators on $\ell^2(\bbZ^d)$ and the dynamics is given by a $\bbZ^d$-action of homeomorphisms on a compact metric space $\Omega$. The compatibility condition is phrased in the terms of the covariance property
\begin{equation}\label{e.covariance}
H_{T^n \omega} = U_n^* H_\omega U_n,
\end{equation}
where $n = (n_1,\ldots,n_d) \in \bbZ^d$ and $U_n$ is the unitary map on $\ell^2(\bbZ^d)$ induced by the translation by $n$ on $\bbZ^d$. The \emph{density of states measure} (DSM) $dk$ is given by the $\mu$-average over $\omega$ of the spectral measure corresponding to the pair $(H_\omega,\delta_0)$,
\begin{equation}\label{e.dos}
\int g \, dk = \int_\Omega \langle \delta_0, g(H_\omega) \delta_0 \rangle \, d\mu(\omega),
\end{equation}
and the \emph{integrated density of states} (IDS) is its accumulation function,
\begin{equation}\label{e.ids}
k(E) = \int \chi_{(-\infty,E]} \, dk .
\end{equation}
By construction, $k$ is a non-decreasing function of $E \in \bbR$ with range $[0,1]$. Moreover, under suitable locality conditions, which will be satisfied for the operators considered in this paper, it can be shown that $k$ is continuous.

It turns out that the almost sure spectrum is given by the set of growth points of the IDS. In other words, it is equal to the topological support of the DSM,
\begin{equation}\label{e.dsmspec}
\mathrm{supp} (dk) = \Sigma = \sigma(H_\omega) \; \text{ for $\mu$-a.e. } \omega \in \Omega.
\end{equation}

This addresses the first issue: the gaps can be labeled by the constant value the IDS takes on each of them. That is, by \eqref{e.dsmspec}, $k$ is constant on each $(g_j^{(-)}, g_j^{(+)})$ and hence one can assign its value to the gap, and again by  \eqref{e.dsmspec}, the values corresponding to different gaps are different. 

The second issue is addressed by showing that each gap label, given by the value of the IDS, must belong to a certain subgroup of $\bbR$ that is at most countable. 

In the general setting, such a countable subgroup of $\bbR$ is given by the range of the normalized trace induced by the ergodic measure, applied to the $K_0$-group associated with the $C^*$-algebra generated by the base dynamics. Let us denote this group by
\begin{equation}\label{e.ellk}
\mathcal{L}_K = \tau^{(\mu)}_*(K_0(C^*(\Omega,T))).
\end{equation}
Thus, for each spectral parameter $E$ in an interior gap $(g_j^{(-)}, g_j^{(+)})$, the value of the IDS at that point must obey $k(E) \in \tau^{(\mu)}_*(K_0(C^*(\Omega,T))) \cap (0,1)$; compare \cite{Be86, Be92}.

In the case $d = 1$, and again assuming suitable locality conditions, there is a second approach. One-dimensionality is used to connect the IDS to the rotation number via Sturm oscillation theory. The latter can then be conveniently determined via the existence of stable and unstable sections for the transfer matrix cocycle for all $E$ outside $\Sigma$. This in turn shows that the rotation number, and hence the IDS, must take values in the range of the Schwartzman homomorphism, and one can use the following label set,
\begin{equation}\label{e.ells}
\mathcal{L}_S = \mathcal{A}^{(\mu)}(C^\#(\tilde \Omega, \tilde T)).
\end{equation}
Here, $(\tilde \Omega, \tilde T)$ denotes the suspension of the dynamics, $C^\#(\tilde \Omega, \tilde T)$ denotes the countable set of homotopy equivalence classes of continuous maps $\phi : \tilde \Omega \to \T$, and $\mathcal{A}^{(\mu)}$ denotes the Schwartzman homomorphism that associates the rotation number with such a homotopy class. See \cite{DF22, DF23b, J86, JONNF} for more details.

Given these general gap labelling theorems, the following questions arise naturally:
\begin{itemize}

\item[(Q1)] When $d = 1$, do we always have $\mathcal{L}_K = \mathcal{L}_S$?

\item[(Q2)] Considering one of the label sets $\mathcal{L} \in \{ \mathcal{L}_K, \mathcal{L}_S \}$ (whenever it is defined), is it a minimal set of labels? That is, is each $\ell \in \mathcal{L} \cap (0,1)$ the label of some interior gap for some operator family with the given $(\Omega,T,\mu)$?

\end{itemize}

Note that as soon as a label does arise for some operator family, it does so for an open set of operator families, as gap boundaries vary continuously. Thus, as the set of labels is countable, genericity issues come down to answering the following question:
\begin{itemize}

\item[(Q3)] If a label $\ell \in \mathcal{L} \cap (0,1)$ is assumed on some interior gap for some operator family, is the set of operator families for which this happens dense?

\end{itemize}

Questions Q2 and Q3 are wide open when $d > 1$, but there are results when $d = 1$. In the latter case, Q1 then becomes interesting as well, and there are some results in that direction as well; compare \cite{PV80, PV80a}.

For these reasons, let us now restrict our attention to the case $d = 1$. Regarding Q1, there is no general result, but there is also no counterexample. That is, $\mathcal{L}_K = \mathcal{L}_S$ has been shown for certain classes of $(\Omega,T,\mu)$, namely by computing both label sets explicitly and observing that they are equal, and there is no known example for which $\mathcal{L}_K \not= \mathcal{L}_S$.

Some of the existing results regarding Q2 and Q3 answer both questions in one fell swoop. Indeed, in order to show that a gap with a given label $\ell \in \mathcal{L} \cap (0,1)$ exists, one starts with a reference operator family, considers an energy $E$ for which $k(E) = \ell$ (which must exist since $k$ is a continuous function with range $[0,1]$) and shows that a suitable small perturbation of the operator family has an open gap (near $E$) with that label. This shows existence, and if one can start with any reference operator family, also denseness. We would call results of this nature \emph{global}.

On the other hand, there are also some results that apply only in the vicinity of a special reference model, and such results are naturally called \emph{local}.

Finally, there are some results that are in between these two cases --- they apply to reference models that belong to a large subset of the space of all models, but not the full space. These will be referred to as \emph{semi-global}.

In this paper we will be interested in establishing global and semi-global results. The general point we wish to make is that it is essential to complement gap labelling theorems with ranges of applicability. The scarcity of the existing global and semi-global results naturally motivates the desire and effort to establish additional ones.

In terms of global results, we start our discussion with a result of Avila-Bochi-Damanik \cite{ABD12} that shows open and dense gap opening for all $\ell \in \mathcal{L}_S$ for Schr\"odinger operators in $\ell^2(\bbZ)$ defined by continuous sampling functions over strictly ergodic base dynamics with a finite-dimensional factor. We prove the analog of this result for Jacobi matrices and extended CMV matrices in $\ell^2(\bbZ)$. The interest in such a result stems from the fact that the three classes of operators that are now covered (Schr\"odinger, Jacobi, CMV) are precisely those classes of self-adjoint and unitary dynamically defined operators in $\ell^2(\bbZ)$ that are of broadest appeal in mathematical physics, spectral theory, and approximation theory. Especially from an inverse spectral theory perspective, the extension from Schr\"odinger to Jacobi and CMV is of fundamental interest.

As a result of semi-global nature, we investigate gap opening for analytic quasi-periodic one-frequency Jacobi and (extended) CMV operators in the subcritical region. Avila's global theory partitions the portion of the energy axis/circle that corresponds to the spectrum into three regions (via the associated cocycle behavior) -- subcritical, critical, and supercritical -- and shows that the critical region is typically empty. There is work by Goldstein and Schlag on gap opening in the supercritical region \cite{GS11}, and we complement this here by addressing the subcritical region.

The organization of the paper is as follows. In Section~\ref{sec.results} we describe our main results in detail. Then we collect in Section~\ref{sec.pre} some basic definitions and known results that we need in the proofs of these results. In particular, we present \cite[Lemma 10]{ABD09}  and \cite[Theorem 3]{ABD12} in Lemma~\ref{lem.ToLocal} and Lemma~\ref{thm.AUH} as $\chi$-valued versions, where $\chi$ can be either $\mathrm{SL}(2,\bbR)$ or $\mathrm{SU}(1,1)$. This allows us to essentially reduce our global results to suitable versions of the projection lemma for Jacobi and CMV matrices. We prove the projection lemma for CMV matrices in Section~\ref{CMVProject} and for Jacobi matrices in Section~\ref{JacobiProject}. The reason for this arrangement is that the proof in the CMV case is more difficult and interesting in itself, and requires new ideas, whereas the extension from the Schr\"odinger to the Jacobi case is comparatively more straightforward. In Section~\ref{mainProof} we use the projection lemmas to  prove our pair of global results, Theorem~\ref{thm.main1} and Theorem~\ref{thm.main3}. Finally, Section~\ref{analyticCase} addresses our pair of semi-global results for subcritical quasi-periodic Jacobi and CMV operators, Theorem~\ref{thm.main7} and Theorem~\ref{thm.main8}.

\section{The Setting and the Main Results}\label{sec.results}

In this section we describe the setting in which we work and the results we will prove. Specifically, we need to specify requirements on the base dynamics $T : \Omega \to \Omega$ and the operators on $\ell^2(\bbZ)$ we generate via iterations of $T$.

\subsection{The Base Dynamics}

We assume that $\Omega$ is a compact metric space and  $T : \Omega \to \Omega$ is a strictly ergodic homeomorphism with respect to an invariant probability measure $\mu$. We assume furthermore that $T$ has a non-periodic finite-dimensional factor \footnote{That is, there is a homeomorphism $T' : \Omega' \to \Omega'$, where $\Omega'$ is an infinite compact subset of some Euclidean space $\bbR^d$, and there is an onto continuous map $h: \Omega \to \Omega'$ such that $h \circ T = T' \circ h.$ }.  In particular, this setting includes the following extensively studied cases:
\begin{itemize}
\item  minimal translations $\omega \mapsto \omega + \alpha$ of the $d$-torus $\T^d$ for any $d\geq 1$;
\item the skew-shift $(\omega_1,\omega_2) \to(\omega_1 + \alpha,\omega_1+\omega_2)$ on $\T^2$, where $\alpha$ is irrational.
\end{itemize}

\subsection{The Operators}

The two classes of operators of interest are given by Jacobi matrices and extended CMV matrices with coefficients generated by the base dynamics $T : \Omega \to \Omega$.

Beginning with the Jacobi case, let $a \in C(\Omega,\bbR_+)$, $b\in C(\Omega,\bbR)$, and set
\begin{equation}\label{eq.JacobiC}
a_n = a(T^n(\omega)), \quad b_n=b(T^n (\omega)). 
\end{equation}
These coefficients define a two-sided Jacobi matrix
$$J=\left(\begin{matrix}\cdots&\cdots&\cdots&\cdots&\cdots&\cdots&\cdots\\
\cdots&a_0 &b_1&a_1&0&\cdots&\cdots\\
\cdots&0&a_1&b_2&a_2&0&\cdots\\
\cdots&0&0&a_2&b_3&a_3&\cdots\\
\cdots&\cdots&\cdots&\cdots&\cdots&\cdots&\cdots\end{matrix}\right),$$
acting on $\ell^2(\bbZ)$ as a bounded self-adjoint operator. As $T$ is strictly ergodic, the spectrum is independent of $\omega \in \Omega$ and may be denoted by $\Sigma$. Whenever we want to emphasize the dependence of $\Sigma$ on the sampling functions, we will write $\Sigma_{a,b}$.

Observe that a Jacobi matrix takes the form of a standard Schr\"odinger operator on $\ell^2(\bbZ)$ in the special case $a(\omega) \equiv 1$. Aside from being more general than Schr\"odinger operators, Jacobi matrices are the more natural objects in inverse spectral theory considerations because there is a nice and well-understood bijection between various explicit classes of (one-sided) Jacobi and explicit classes of probability measures on the real line. Via the spectral theorem it then follows that one-sided Jacobi matrices are the natural building blocks of general self-adjoint operators. Two-sided Jacobi matrices are the more natural object in the dynmically defined setting and much of their analysis goes through the study of the one-sided case anyway. See \cite{Simon3} for more details.

Turning now to the CMV case, denote the open unit disc in the complex plane by $\bbD$, that is, $\bbD=\{z\in\mathbb{C}:|z|<1\}$. Given $v \in C(\Omega, \bbD)$, we can set
\begin{equation}\label{eq.VerblunskyC}
\alpha_n=v(T^{n}(\omega)) \quad \omega \in \Omega
\end{equation}
and view them as the Verblunsky coefficients defining an extended CMV matrix
\begin{equation*}
\mathcal{E}=\left(
\begin{matrix}
\cdots&\cdots&\cdots&\cdots&\cdots&\cdots&\cdots\\
\cdots&-\overline{\alpha}_0\alpha_{-1}&\overline{\alpha}_1\rho_{0}&\rho_1\rho_0&0&0&\cdots&\\
\cdots&-\rho_0\alpha_{-1}&-\overline{\alpha}_1\alpha_{0}&-\rho_1\alpha_0&0&0&\cdots&\\
\cdots&0&\overline{\alpha}_2\rho_{1}&-\overline{\alpha}_2\alpha_{1}&\overline{\alpha}_3\rho_2&\rho_3\rho_2&\cdots&\\
\cdots&0&\rho_2\rho_{1}&-\rho_2\alpha_{1}&-\overline{\alpha}_3\alpha_2&-\rho_3\alpha_2&\cdots&\\
\cdots&0&0&0&\overline{\alpha}_4\rho_3&-\overline{\alpha}_4\alpha_3&\cdots&\\
\cdots&\cdots&\cdots&\cdots&\cdots&\cdots&\cdots
\end{matrix}\right),
\end{equation*}
where $\rho_n=\sqrt{1-|\alpha_n|^2}$. As before, the strict ergodicity of $T$ ensures that the spectrum is independent of $\omega \in \Omega$ and may be denoted by $\Sigma$. Whenever we want to emphasize the dependence of $\Sigma$ on the sampling function, we will write $\Sigma_{v}$.

Comments similar to the Jacobi case apply. There is a bijection between one-sided CMV matrices and probability measures on the unit circle. Again by the spectral theorem, a general unitary operator can then be written as a direct sum of operators that are (unitarily equivalent to) one-sided CMV matrices. The two-sided case is again more natural in the dynamical setting. See \cite{Simon1, Simon2} for more details.

In summary, the two classes of operators we consider are canonical in the self-adjoint and unitary settings, respectively. Our focus is on the case of dynamically defined coefficients and the structure of their spectrum.

\subsection{Gap Labelling via the Schwartzman Homomorphism}

Gap labelling via the Schwartzman homomorphism rests on three pillars. First, the solutions to the generalized eigenvalue equation for a dynamically defined Jacobi or CMV matrix can be described via a cocycle. Second, Johnson's theorem states that the complement of $\Sigma$ is characterized by the uniform hyperbolicity of the associated cocycle. This results in the existence of continuous (unstable and stable) sections for the cocycle. Third, the Schwartzman homomorphism applied to (a continuous interpolation of) such a continuous section then determines the rotation number, which is directly related to the value of the IDS via oscillation theory.

Let us make this explicit in the cases at hand. In the Jacobi case, the generalized eigenvalue equation is given by
$$
(Ju)_n = a_n u_{n+1} + b_n u_n + a_{n-1}u_{n-1} = Eu_n
$$
for a spectral parameter (energy) $E \in \bbR$. A sequence $\{u_n\}$ is a solution if and only if it satisfies 
\begin{equation}\label{eq.transfer}
\begin{bmatrix}u_{n+1}\\a_nu_{n}\end{bmatrix}=\frac{1}{a_n}\begin{bmatrix}E-b_n&-1\\a_n^2&0\end{bmatrix}\begin{bmatrix}u_n\\a_{n-1}u_{n-1}\end{bmatrix}.
\end{equation}
Since the coefficients $\{ a_n, b_n \}$ are dynamically defined via \eqref{eq.JacobiC}, the transfer matrix appearing in \eqref{eq.transfer} is as well, namely via the $\mathrm{SL}(2,\bbR)$-cocycle
\begin{equation}\label{eq.JacobiCocycle}
(T,A_{a,b}) : \Omega \times \bbR^2 \to \Omega \times \bbR^2 , (\omega , v) \mapsto \Big( T(\omega), A_{a,b}^{(E)}(\omega) v \Big)
\end{equation}
with
\begin{equation}\label{eq.JacobiMap}
A_{a,b}^{(E)}(\omega) = \frac{1}{a(\omega)}\begin{bmatrix} E - b(\omega) & -1 \\ a(\omega)^2 & 0 \end{bmatrix}.
\end{equation}
Since $A_{a,b}^{(E)}(\omega) \in \mathrm{SL}(2,\bbR)$, one may projectivize \eqref{eq.JacobiCocycle} and consider instead
\begin{equation}\label{eq.JacobiCocycleProj}
(T,A_{a,b}) : \Omega \times \bbR\bbP^1 \to \Omega \times \bbR\bbP^1 , (\omega , [v]) \mapsto \Big( T(\omega), [A_{a,b}^{(E)}(\omega) v] \Big).
\end{equation}
Johnson's theorem now states that for every $E \in \bbR \setminus \Sigma$, there is a continuous invariant section for \eqref{eq.JacobiCocycleProj}. Since \eqref{eq.JacobiMap} is homotopic to the identity, we can extend \eqref{eq.JacobiCocycle} to a continuous flow on the mapping torus $(\tilde \Omega, \tilde T)$ (a.k.a.\ suspension) of $(\Omega,T)$, and the same is true for the continuous invariant section. Thus, we obtain an element of $C(\tilde \Omega,\T)$ if we identify $\bbR\bbP^1$ and $\T$ in the natural way. This in turn allows us to apply the Schwartzman homomorphism to its homotopy class and obtain the rotation number $\rho(E)$, which by Sturm oscillation theory is equal to $1 - k(E)$; compare \cite{DFZ23}.

%We call the matrix $A=\frac{1}{a}\begin{bmatrix}E-b&-1\\a^2&0\end{bmatrix}$ Jacobi cocycle. In particular, if the coefficients $\{a_n,b_n\}$ are given by continuous sampling along an orbit:
%$a_{n}=a(f^{n}x), b_n=b(f^nx)$, where $a,b\in C(X,\bbR)$, then we denote the cocycle by $(f,A_{a,b})$. In the present paper, we will assume that the function $a$ is fixed and consider the perturbations on $b$. Therefore, we supress $a$ from the notations and denote the cocycle by $(f,A_b)$. If $f$ is ergodic, then denote the almost sure spectrum by $\sigma(J_{f,b})$.

\medskip

Turning to the CMV case, denote by $z\in\partial\bbD=\{z\in\mathbb{C}:|z|=1\}$ the spectral parameter and consider the generalized eigenvalue equation 
$$
\mathcal{E}u=zu.
$$
Then solutions can be described in terms of the {\it normalized Szeg\H{o} cocycle map} given by  
\begin{equation}\label{eq.SzegoMap}
S_v^{(z)}(\omega) = \frac{z^{-\frac{1}{2}}}{\sqrt{1-|v(\omega)|^2}} \begin{bmatrix} z & -\overline{v(\omega)} \\ -v(\omega) & 1 \end{bmatrix},
\end{equation}
which takes values in $\mathrm{SU}(1,1)$. Recalling that 
\begin{equation}\label{eq.su11sl2r}
M^{-1} \mathrm{SU}(1,1) M = \mathrm{SL}(2,\bbR) \text{ with } M = \frac{1}{1+i}\begin{bmatrix}1&-i\\1&i\end{bmatrix},
\end{equation}
we can again view the transfer matrices as being generated by an $\mathrm{SL}(2,\bbR)$-valued cocycle, which allows us to proceed as above with defining the associated rotation number $\rho(z)$; compare \cite{GJ96}.

The upshot is that in both cases, gaps of $\Sigma$ can be labeled with elements of the range $\mathcal{L}_S$ of the Schwartzman homomorphism, intersected with the range of the rotation number/IDS.

\subsection{The Main Results}

In this subsection we state our main results. We begin with global results in the continuous category for which we assume the base dynamics to be as above.

\begin{theorem}[Gap opening for Jacobi matrices]\label{thm.main1}
For every $a\in C(\Omega,\bbR_+)$, $b\in C(X,\bbR)$, and $\ell\in\mathcal{L}_S \cap (0,1)$, there exists a continuous path $b_t\in C(\Omega,\bbR), 0\leq t\leq 1$ with $b_0=b$ such that for every $t>0$, the spectrum $\Sigma_{a,b_t}$ has an open gap with label $\ell.$
\end{theorem}

This implies that for every $a\in C(\Omega,\bbR_+)$ and every $\ell\in\mathcal{L}_S \cap (0,1)$, the set of $b\in C(\Omega,\bbR)$ such that $\Sigma_{a,b}$ has an open gap with label $\ell$ is open and dense. Since $\mathcal{L}_S$ is countable, we have the following as a direct consequence:
\begin{coro}\label{coro.main2}
For every $a\in C(\Omega,\bbR_+)$, there is a dense $G_\delta$ set of $b\in C(\Omega,\bbR)$ such that $\Sigma_{a,b}$ has all gaps allowed by the Gap Labelling Theorem open.
\end{coro}

The analogous results for extended CMV matrices read as follows:

\begin{theorem}[Gap opening for extended CMV matrices]\label{thm.main3}
For every $v \in C(\Omega,\bbD)$  and $\ell\in\mathcal{L}_S \cap (0,1)$, there exists a continuous path $v_t\in C(\Omega,\bbD), 0\leq t\leq 1$ with $v_0=v$ such that for every $t>0$, $\Sigma_{v_t}$ has an open gap with label $\ell$. 
\end{theorem}

\begin{coro}\label{coro.main4}
There is a dense $G_{\delta}$ set of $v \in C(\Omega,\bbD)$ such that  $\Sigma_{v}$ has all gaps allowed by the Gap Labelling Theorem open.
\end{coro}

\begin{remark}
(a) In the Schr\"odinger case, results of this nature were obtained by Avila-Bochi-Damanik in \cite{ABD12}, providing a negative answer to one of Bellissard's questions, ``whether mixing avoids Cantor spectrum" \cite[Problem 5]{Be82}; compare \cite[Appendix E]{ABD12}.
\\[1mm]
(b) The work of \cite{ABD12} proceeds in two steps. First, uniform hyperbolicity is produced for a perturbed $\mathrm{SL}(2,\bbR)$-valued cocycle that retains the rotation number. Second, a projection lemma then explains how to model this with cocycles of Schr\"odinger form. The first part will be unchanged for our consideration of the cocycles describing solutions to the generalized eigenvalue equation for Jacobi and extended CMV matrices, as explained above. The second part, however, is model-dependent and we will have to establish suitable projection lemmas for the two cases. 
\\[1mm]
(c) These results provide new examples related to Bellissard's question and enrich the family of  operators for which the simultaneous opening of all gaps can be guaranteed in the generic sense. Moreover, it also gives a positive answer to another question of Bellissard, namely ``whether Cantor spectrum is generic whenever it is not precluded by the Gap Labeling Theorem" \cite [Question 6]{Be86}, in the settings of Jacobi and extended CMV matrices.
\end{remark}

\medskip

We now turn to our semi-global result in the subcritical regime of analytic one-frequency quasi-periodic Jacobi and extended CMV matrices. Thus, for this discussion, we consider $\Omega = \T$, $T \omega = \omega + \alpha$ with $\alpha$ irrational, and analytic sampling functions $a,b,v$. Note that an irrational rotation of the circle is strictly ergodic, with Lebesgue measure as the unique invariant probability measure on $\T$.

Let us begin by recalling results known in the Schr\"odinger case ($a \equiv 1$). The \emph{Ten Martini Problem} asks whether for $b(\omega) = 2 \lambda \cos (2 \pi \omega)$ with $\lambda > 0$, $\Sigma$ is a Cantor set. This is by now known to be the case for all parameter values; see Avila-Jitomirskaya \cite{AJ1} and references therein.  The stronger version, where all gaps allowed by the Gap Labelling Theorem are shown to be open, is known as the \emph{Dry Ten Martini Problem}. This statement is now known in almost all cases; compare \cite{CEY, Pui04, AJ2, AYZ}.

Very recently, Ge-Jitomirskaya-You \cite{GJY} showed that ``having all gaps are open" is a stable phenomenon in the neighborhood of almost Mathieu operators with coupling constant $0<|\lambda|<1$. Together with the progress of Xu-Ge-Wang \cite{GWX} for sufficiently large coupling constants, having all gaps open is a rather robust property. It would be interesting to see if this is true for general analytic potentials; see \cite{GJY24}.

To discuss the general analytic quasi-periodic one-frequency case, let us recall Avila's classification of cocycle behavior and the resulting partition of the spectrum. Given that the complement of the spectrum consists of those spectral parameters for which the cocycle is uniformly hyperbolic by Johnson's theorem, the points in the spectrum are partitioned into subcritical, critical, and supercritical values, based on the behavior of the associated cocycle. The supercritical case simply corresponds to the case of a positive Lyapunov exponent. Points in the spectrum where the Lyapunov exponent vanishes, are either subcritical or critical, depending on whether the cocycle retains a subexponential upper bound after complexification or not; see \cite{A} for details. Avila also shows in \cite{A} that for typical sampling functions, there are no critical values of the spectral parameter and hence the spectrum can be viewed as the disjoint union of the subcritical part and the supercritical part. Gaps of $\Sigma$ in the supercritical part were studied by Goldstein and Schlag in \cite{GS11}. Here we address the subcritical region and prove the following result, which is of semi-global nature:

%In \cite{ABD12}, the authors mentioned that it is unclear if the $C$ topology is an ``artificial" condition to have generic gap opening due to the lack of evidence in the analytic topology. Despite some local results \cite{Eli92, Puig006, LDZ22}, a semi-local version of result should be expected via almost reducibility. Here we provide a different approach to get the following results: 

\begin{theorem}\label{thm.main7} Let $\alpha$ be Diophantine and $a\in C^{\omega}(\T,\mathbb{R}_{+})$. Then, for a generic $b\in C^{\omega}(\T,\bbR)$, the corresponding quasi-periodic Jacobi matrix $J_{a,b}$ has all its spectral gaps open in the subcritical region. 
\end{theorem}

\begin{remark}
For Diophantine $\alpha$ and $a\equiv 1$, the theorem implies that for a generic  $b\in C^{\omega}(\T,\mathbb{R})$, the corresponding Schr\"odinger operator has all its spectral gaps open in the subcritical region. This can also be approached by a combination of Avila's almost reducibility conjecture \cite{Avila10, Avila24} with Puig's argument \cite{Puig006}. The same result for Liouville $\alpha$ was recently obtained by Li-Zhou \cite{LZ24}.
\end{remark}
For quasi-periodic CMV matrices, we have the following:

\begin{theorem}\label{thm.main8}
Let $\alpha$ be Diophantine. Then, for a generic $(\lambda,h)\in (0,1)\times C^{\omega}(\T,\bbR)$, the  extended CMV matrix generated by the sampling function $v=\lambda e^{ih}$ has all its spectral gaps open in the subcritical region.
\end{theorem}

\section{Preliminaries}\label{sec.pre}

\subsection{Cocycle Dynamics}
Let $\chi$ be either $\mathrm{SL}(2,\bbR)$ or $\mathrm{SU}(1,1)$ and $A\in C(\Omega,\chi)$, where $\Omega$ is a compact metric space, and let $T: \Omega \to \Omega$ be a homeomorphism.
Given such a pair $(T,A)$, referred to as a ($\chi$-)cocycle, define
$$
A^n(\omega)=\left\{\begin{aligned}
&A(T^{n-1}(\omega))\cdot A(T^{n-2}(\omega))\cdots A(\omega)~ n\geq 0;\\
&A^{-1}(T^n(\omega))\cdot A^{-1}(T^{n-1}(\omega))\cdots A^{-1}(T(\omega))~ n<0.
\end{aligned}\right.$$
In the present paper, we will adopt different notations for cocycles, depending on what dependence of parameters we want to make explicit. We hope the readers will not be confused.

\begin{definition}\label{def.UH}
A cocycle $(T,A)$ is said to be {\it uniformly hyperbolic} if there exist constants $c$ and $\lambda>1$ such that $\Vert A^n(x)\Vert> c\lambda^n$ for every $\omega$ and $n>0$.
\end{definition}
We say two cocycles $(T,A)$ and $(T,B)$ are {\it conjugate} if there exists a conjugacy $Z\in C(\Omega,\chi)$ such that 
$$
B(\omega)=Z(T(\omega))A(\omega)[Z(\omega)]^{-1}.
$$ 
Note that uniform hyperbolicity is a conjugacy invariant.

For a Jacobi matrix $J$ with dynamically defined coefficients as in \eqref{eq.JacobiC} and the corresponding Jacobi cocycles $(T,A_E)$, in which we make the dependence of the spectral parameter $E \in \bbR$ explicit, $E \in \sigma(J)$ if and only if the cocycle $(T,A_E)$ is uniformly hyperbolic. That is, let $\mathcal{UH}$ be the set of all uniformly hyperbolic cocycles, then \begin{equation}\label{eq.JacobiUH}
\sigma(J)=\{E\in\bbR:(T,A_E)\notin \mathcal{UH}\}.
\end{equation}

Analogously, for an extended CMV matrix $\mathcal{E}$ with dynamically defined coefficients as in \eqref{eq.VerblunskyC} and the associated Szeg\H{o} cocycle $(T,A_z)$, where $z$ is the spectral parameter, we have $z\in\sigma(\mathcal{E})$ if and only if $(T,A_z)$ is not uniformly hyperbolic, that is
\begin{equation}\label{eq.CMVUH}
\sigma(\mathcal{E})=\{z\in\partial\bbD: (T,A_z)\notin\mathcal{UH}\}.
\end{equation}

\subsection{Fibered Rotation Number}

The following definition is for $\mathrm{SL}(2,\bbR)$-valued cocycles and can be translated to a $\mathrm{SU}(1,1)$ version immediately by the isomorphism \eqref{eq.su11sl2r}. Alternatively, one can also compare the definition for Szeg\H{o} cocycles in \cite[Chapter 8.3]{Simon1}.

\begin{definition}[Rotation number]
Let $(T,A)$ be an $\mathrm{SL}(2,\bbR)$ valued cocycle that is homotopic to a constant. Then one can define a homeomorphism $F : \Omega \times \bbR \to \Omega \times \bbR$ by $F(\omega,t) = (T(\omega),F_2(\omega,t))$, where $F_2(\omega,t)$ is the argument of $A(\omega)\cdot [\cos(2\pi t),\sin(2\pi t)]^{t}$ by identifying $\bbR^2$ with $\mathbb{C}$ in the usual way. $F_2(\omega,t)$ is called a {\it lift} of $(T,A)$. The lift is not unique since for any continuous function $\phi : \Omega \to \mathbb{Z}$, $F_2(\omega,t)+\phi(\omega)$ is also a lift.  Let $F^n_2(\omega,t)$ be the second component of $F^n(\omega,t)$, then 
the limit 
$$
\rho_{T,A} = \lim_{n\to\infty} \frac{F_2^n(\omega,t)-t}{n} 
$$ 
exists uniformly and is independent of $(\omega,t)$. It is called the {\it fibered rotation number} of the cocycle $(T,A)$ \cite{Her83}. Sometimes we will also use notations $\rho_{T,b}(E)$ and $\rho_{T,v}(z)$ to make the dependence of the rotation numbers on sampling functions $b, v$ and spectral parameters $E,z$ explicit whenever the cocycle maps are clear under the context.
\end{definition}

\begin{remark}
The literature contains various ways of introducing and discussing the concept of the rotation number. In particular, the definition just given reproduces the rotation number obtained via the Schwartzman homomorphism, up to a factor $2$ (due to the use of the unit circle as opposed to the real projective line). We work with the definition given here to conform with \cite{ABD12}, but care needs to be exercised regarding the factor $2$.
\end{remark}

\medskip

Let $\Omega$ be a compact metric space with at least three points, for any $A \in C(\Omega,\chi)$ and $K \subset \Omega$, let $C_{A,K}(\Omega,\chi)$ be the set of all $B$ such that $B(\omega)=A(\omega)$ for $\omega$ outside $K$ and denote by $\mathcal{W}_{A,K}$ a neighborhood of $A$ in $C_{A,K}(\Omega,\chi)$. We need the following result, which is \cite[Lemma 10]{ABD09} (parsed as a $\chi$-valued version with ease), to convert the problem to a local one:

\begin{lemma}\label{lem.ToLocal}
Let $V\subset \Omega$ be any nonempty open set, and let $A\in C(\Omega,\chi)$ be arbitrary. Then there exist an open neighborhood $\mathcal{W}_{A,V}\subset C(\Omega,\chi)$ of $A$ and continuous maps 
$$
\Phi = \Phi_{A,V}: \mathcal{W}_{A,V}\to C_{A,\overline{V}}(\Omega,\chi)
$$
and 
$$
\Psi=\Psi_{A,V}: \mathcal{W}_{A,V}\to C(\Omega,\chi)
$$
such that 
$$
\Psi(B)(T(\omega))\cdot B(\omega) \cdot [\Psi(B)(\omega)]^{-1} = \Phi(B)(\omega)
$$
$$
\Phi(A) = A \quad \text{and} \quad \Psi(A) = id.
$$
\end{lemma}

The following result (see \cite[Theorem 3]{ABD12}) is the key to opening gaps:

\begin{lemma}[Accessibility by uniformly hyperbolic cocycles]\label{thm.AUH}
Suppose $A\in C(\Omega,\chi)$ is homotopic to a constant and obeys $2\rho_{T,A}\in \mathcal{L}_S$. Then there exists a continuous path $A_t\in C(\Omega,\chi)$, $0\leq t\leq 1$ with $A_0=A$ and such that $A_t\in\mathcal{UH}$ for every $t>0$. In particular, $\rho_{T,A_t}$ is independent of $t$. 
\end{lemma}

\section{The Projection Lemma for CMV Matrices}\label{CMVProject}

Since any $Z\in\mathrm{SU}(1,1)$ can be written as $\begin{bmatrix}a&b\\\bar{b}&\bar{a}\end{bmatrix}, |a|^2-|b|^2=1$, let $v,\theta,\varphi$ be determined by  $$a=\frac{1}{\sqrt{1-v^2}}e^{i\theta}, b=\frac{v}{\sqrt{1-v^2}}e^{-i\varphi},$$
therefore, \begin{equation}\label{eq.SU11}Z=\frac{1}{\sqrt{1-v^2}}\begin{bmatrix}e^{i\theta} &ve^{-i\varphi}\\ve^{i\varphi}&e^{-i\theta}\end{bmatrix}.\end{equation} 
This implies that every $Z\in\mathrm{SU}(1,1)$ can be parametrized as $(\theta, \varphi, v),\theta,\varphi\in\bbR, v\in[0,1).$
Recall a Szeg\H{o} cocycle takes exactly the same form. Now let's fix $\theta\in [0,\pi]$ and denote
\begin{equation}\label{eq.szegoM}\mathcal{S}_{\theta}=\left\{\frac{1}{\sqrt{1-v^2}}\begin{bmatrix}e^{i\theta}&ve^{-i\varphi}\\ve^{i\varphi}&e^{-i\theta}\end{bmatrix}:\varphi\in\bbR, v\in[0,1)\right\}.\end{equation}
Notice that by the definition, an element of $\mathcal{S}_\theta$ has upper left entry  $\frac{e^{i\theta}}{\sqrt{1-v^2}}$ with $\theta$ being a constant, which leads to a difference when it is a function in \eqref{eq.SU11}.

The following lemma is the central part of this section and is the only model-dependent part of the proof.

\begin{lemma}[Projection Lemma: CMV case]\label{lem.project}
Let $T: \Omega \to \Omega$ be a minimal homeomorphism of a compact metric space with at least three points, and let $A\in C(\Omega,\mathcal{S}_{\theta})$. Then there exist a neighborhood $\mathcal{W}\subset C(\Omega,\mathrm{SU}(1,1))$ of $A$ and continuous maps 
$$
\Phi = \Phi_A: \mathcal{W} \to C(\Omega,\mathcal{S}_\theta)\quad \text{and}\quad \Psi=\Psi_A: \mathcal{W} \to C(\Omega,\mathrm{SU}(1,1))
$$
such that for any $B \in \mathcal{W}$,
$$
\Psi(B)(T(\omega)) \cdot B(\omega) \cdot [\Psi(B)(\omega)]^{-1} = \Phi(B)(\omega)
$$
$$
\Phi(A) = A,~\Psi(A)=id.
$$
In particular, an $\mathrm{SU}(1,1)$-valued perturbation of a Szeg\H{o} cocycle corresponding to parameter $\theta$ is conjugate to an $\mathcal{S}_\theta$ valued perturbation.
\end{lemma} 

This result is an $\mathrm{SU}(1,1)$ and Szeg\H{o} analogue of \cite[Lemma 1.3]{ABD12}; however, we would like to mention that representing a general $\mathrm{SU}(1,1)$ matrix as a product of $\mathcal{S}_{\theta}$-valued matrices is in general impossible. The main obstruction is the loss of freedom of the upper-left elements of $\mathcal{S}_{\theta}$ valued matrices as $\theta$ is fixed.

Lemma~\ref{lem.project} is much weaker than \cite[Lemma 11]{ABD09}; in particular, after the completion of \cite{ABD12}, the inverse map is well defined for any $\mathrm{SL}(2,\bbR)$ valued perturbations with desired smoothness. But in the actual applications, the projection map only needs to be applied to $\mathrm{SL}(2,\bbR)$ matrices that are close enough to Schr\"odinger cocycles. This is the reason that Lemma~\ref{lem.project} is sufficient for our purpose.

Due to Lemma~\ref{lem.ToLocal}, to prove Lemma~\ref{lem.project}, it is sufficient to prove the following local version:

\begin{lemma}\label{lem.KeyLem}
Let $K\subset \Omega$ be a compact set such that $K \cap T(K) = \emptyset$ and $K \cap T^2(K) = \emptyset$. Let $A \in C(\Omega,\mathcal{S}_\theta)$. Then there exist an open neighborhood $\mathcal{W}_{A,K}\subset C_{A,K}(\Omega,\mathrm{SU}(1,1))$ of $A$ and continuous maps 
$$
\Phi = \Phi_{A,K} : \mathcal{W}_{A,K} \to C(\Omega,\mathcal{S}_\theta) \quad \text{and}\quad \Psi=\Psi_{A,K}:\mathcal{W}_{A,K} \to C(\Omega,\mathrm{SU}(1,1))
$$
such that for $B\in \mathcal{W}_{A,K}$,
$$
\Psi(B)(T(\omega)) \cdot B(\omega) \cdot [\Psi(B)(\omega)]^{-1} = \Phi(B)(x)
$$
$$
\Phi(A)=A,~\Psi(A)=id.
$$
\end{lemma}

%The proof follows the similar outline of \cite{ABD09}, however, the Szeg\H{o} cocycle maps are more complicated than Schr\"odinger cocycles. And the projection is more subtle compared to projecting into the set of Schr\"odinger cocycles.

\begin{proof}
Let $A_1,A_2,A_3\in\mathcal{S}_{\theta}$ be parameterized in the following way:
\begin{equation}\label{eq.parameterization}
\begin{aligned}
&A_1=\frac{1}{\sqrt{1-v^2}}\begin{bmatrix}e^{i\theta}&ve^{-i\varphi_1}\\ve^{i\varphi_1}&e^{-i\theta}\end{bmatrix},~\varphi_1\in\bbR,\\
&A_2=\frac{1}{\sqrt{1-v_2^2}}\begin{bmatrix}e^{i\theta}&ve^{-i\varphi}\\ve^{i\varphi}&e^{-i\theta}\end{bmatrix},~v_2\in[0,1),\\
&A_3=\frac{1}{\sqrt{1-v^2}}\begin{bmatrix}e^{i\theta}&ve^{-i\varphi_3}\\ve^{i\varphi_3}&e^{-i\theta}\end{bmatrix},~\varphi_3\in\bbR,
\end{aligned}
\end{equation}
where $v,\varphi$ are any constants, that is, we take the triple $(\varphi_1,v_2,\varphi_3)$ as our undetermined parameters.
Let us consider the following injective map:
$$
\eta : \mathcal{S}_{\theta}^3 \to \mathrm{SU}(1,1),
$$
$$
(A_1,A_2,A_3)\to A_3\cdot A_2\cdot A_1.
$$
A direct computation shows $A_3A_2A_1=\frac{1}{\sqrt{1-v_2}(1-v^2)}\begin{bmatrix}a&\bar{b}\\b&\overline{a}\end{bmatrix}$, where 
$$
\begin{aligned}&a=e^{3i\theta}+\left(vv_2e^{i(\varphi-\varphi_1+\theta)}+vv_2e^{i(\theta-\varphi+\varphi_3)}+v^2e^{i(\varphi_3-\theta-\varphi_1)}\right),\\
&b=ve^{i(\varphi_1+2\theta)}+v_2e^{i\varphi}+ve^{i(\varphi_3-2\theta)}+v^2v_2e^{i(\varphi_1+\varphi_3-\varphi)}.\end{aligned}
$$
Comparing \eqref{eq.SU11} with \eqref{eq.szegoM}, the perturbations of the off-diagonal element of \eqref{eq.SU11} can be viewed as perturbations of the off-diagonal element of \eqref{eq.szegoM}. The main difficulty is that the upper-left entry of the element of \eqref{eq.szegoM} has less freedom. One can easily add a perturbation to it and kick it out of $\mathcal{S}_\theta$. Now let us focus on the analysis of the upper-left element of the product $A_3A_2A_1$.

For fixed $\theta,\varphi,v$, denote the region formed by the values of $a$ by $G(v_2)$. Then it is clear from \eqref{eq.ulRange} below that $G(v_2)$ is continuous in $v_2$. More importantly, as $v_2$ increases, the outer boundaries expand in a monotonic way, but we still need to address the inner boundaries, which disappear as $v_2$ hits a certain threshold.  This creates the wiggle room for $\eta$ to set up a local homeomorphism between $\mathcal{S}_\theta$ and $\mathrm{SU}(1,1)$.

To make this precise, let $g:(\varphi_1,v_2,\varphi_3)\to\mathbb{C}$ be 
\begin{equation}\label{eq.ulRange}
vv_2e^{i(\varphi-\varphi_1+\theta)}+vv_2e^{i(\theta-\varphi+\varphi_3)}+v^2e^{i(\varphi_3-\theta-\varphi_1)}.\end{equation} 
The following lemma visualizes the range of $g$ and shows that the inner boundaries will not be a problem.

\begin{lemma}
For any $v\leq v_2<1$, $\mathrm{Ran}(g)$ is simply connected, that is, there is no hole in $\mathrm{Ran}(g)$.
\end{lemma}

\begin{proof}
For any $v_2$, it follows from periodicity that $\mathrm{Ran}(g)$ is closed. It suffices to show that the range of $e^{i\varphi_1}+e^{i\varphi_3}+\lambda e^{i(\varphi_1+\varphi_3)},$ where $0<\lambda\leq 1,$ is simply connected. 
\begin{center}

\begin{tabular}{ccc}

\begin{tikzpicture}
\draw (-2.5,0) -- (2.5,0);
\draw (0,-2.5) -- (0,2.5);
%\draw (0,0) circle (1);
\draw [blue] (0.707,0.707) circle (1.4);
\draw (-0.707,-0.707) circle (0.7368);
\end{tikzpicture}
&\begin{tikzpicture}
\draw (-2.5,0) -- (2.5,0);
\draw (0,-2.5) -- (0,2.5);
%\draw (0,0) circle (1);
\draw [blue] (-0.707,0.707) circle (0.7368);
\draw (0.707,-0.707) circle (1.4);
\end{tikzpicture}
&\begin{tikzpicture}
\draw (-2.5,0) -- (2.5,0);
\draw (0,-2.5) -- (0,2.5);
\draw [blue] (-1,0) circle (0.5);
\draw (1,0) circle (1.5);

%\draw (-1,0) circle ()
\end{tikzpicture}
\\
$\lambda=1/2,\varphi_1=\pi/4$&$\lambda=1/2,\varphi_1=3\pi/4$&$\lambda=1/2,\varphi=\pi$\\
\end{tabular}
\end{center}
Let $D(z,r)$ be the disc centered at $z\in\mathbb{C}$ with radius $r$.
For each fixed $\varphi_1$, the set $\{e^{\varphi_1}+e^{i\varphi_3}(1+\lambda e^{i\varphi_1}),\varphi_3\in [0,2\pi]\}$ represents a circle $\partial D(e^{i\varphi_1},r(\varphi_1))$, where $r(\varphi_1)=\sqrt{1+\lambda^2+2\lambda\cos\varphi_1}$. Note also $r(\varphi_1+\pi)=\sqrt{1+\lambda^2-2\lambda \cos\varphi_1}$, and \begin{equation}\label{eq.intersectCond}r(\varphi_1)+r(\varphi_1+\pi)=\sqrt{1+\lambda^2+2\lambda\cos\varphi_1}+\sqrt{1+\lambda^2-2\lambda\cos\varphi_1}\geq 2,
\end{equation}
where equality happens if and only if $\varphi_1=0,\pi$. Starting with any initial $\varphi_1\neq 0, \pi$, $D(e^{i\phi_1},r(\varphi_1))\cap D(e^{i(\varphi_1+\pi)},r(\varphi_1+\pi))$ is non empty due to \eqref{eq.intersectCond}. Therefore the fundamental group of the plane separated by the two circles is isomorphic to $\mathbb{Z}^3.$ Letting $\varphi_1$ go through $\pi$ and $0$, the initial disc   $D(e^{i\varphi_1},r(\varphi_1))$ will be erased, that is, $\cap_{\varphi_1\in [0,\pi]}D(e^{i\varphi_1},r(\varphi_1))=\emptyset$ and the intersection is contractible, that is, $\cap_{\varphi_1\in [0,\pi]} D(e^{i\varphi_1},r\varphi_1)\cap D(e^{i(\varphi+\pi)},r(\varphi+\pi))=\emptyset$. This shows that for $\lambda\leq 1$, any loop in $\mathrm{Ran}(g)$ is contractible. 
\end{proof}

\begin{remark}
When $\lambda>1$, the circles corresponding to $\varphi_1=0$ and $\varphi_1=\pi$ are inscribed, therefore leave a hole in $\mathrm{Ran}(g)$.
\end{remark}

Now $G(v_2)$ is a monotonically increasing simply connected region for $v_2\geq v$ in $\mathbb{C}$ as $v_2$ increases. The following pictures show the (rescaled) range of $g$:

\begin{center}
\begin{tabular}{c c c c}
\includegraphics[width=34mm, height=34mm]{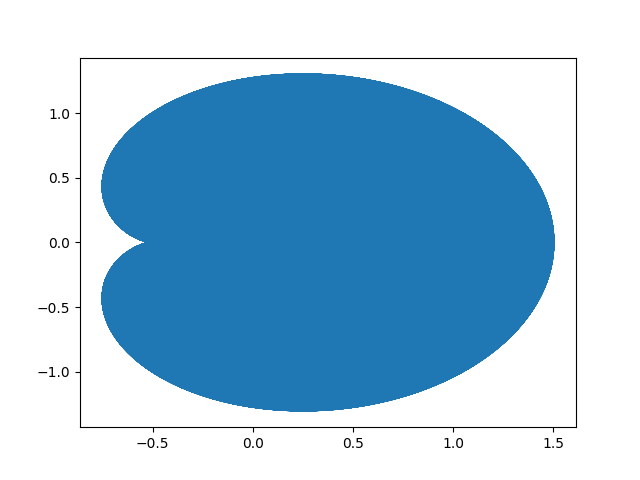}&\includegraphics[width=34mm, height=34mm]{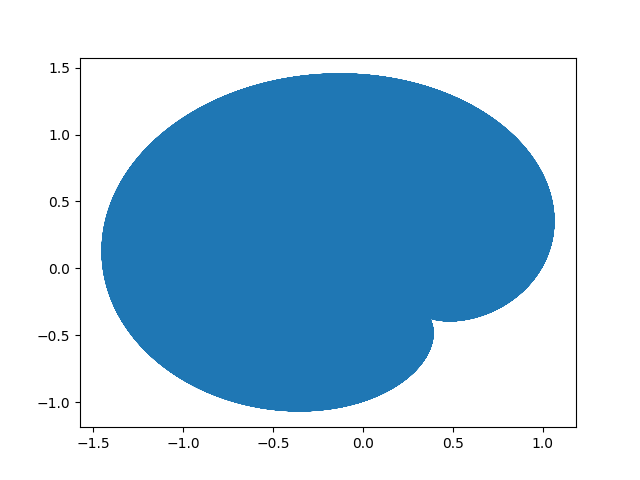}&\includegraphics[width=34mm, height=34mm]{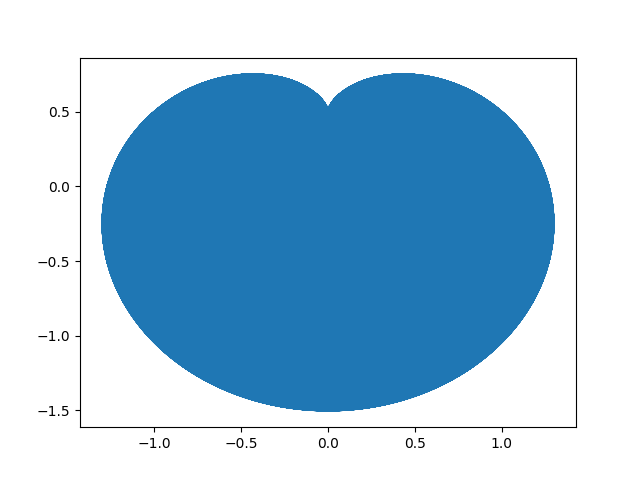}&\includegraphics[width=34mm, height=34mm]{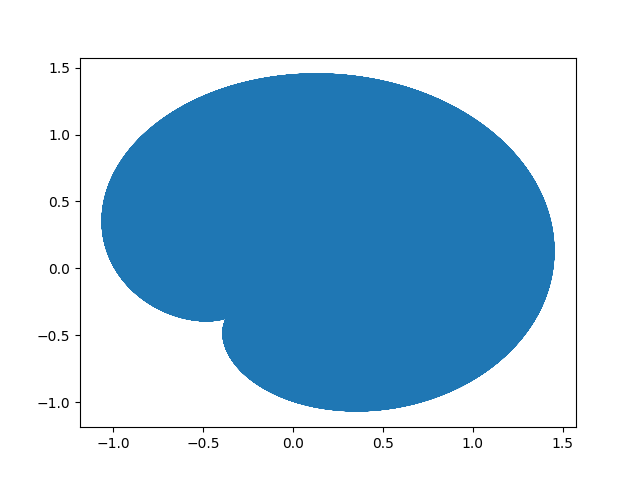}\\
$\theta=0$&$\theta=\pi/4$&$\theta=\pi/2$&$\theta=3\pi/4$\\
\includegraphics[width=34mm, height=34mm]{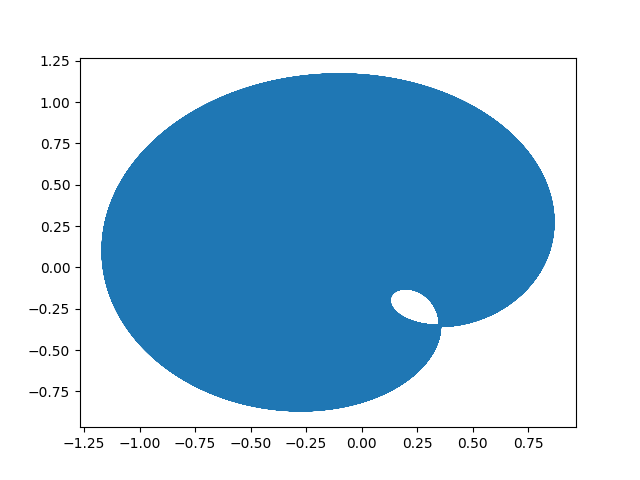}&\includegraphics[width=34mm, height=34mm]{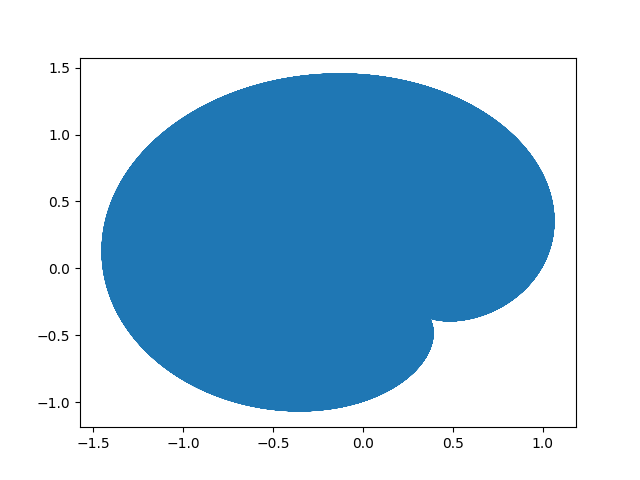}&\includegraphics[width=34mm, height=34mm]{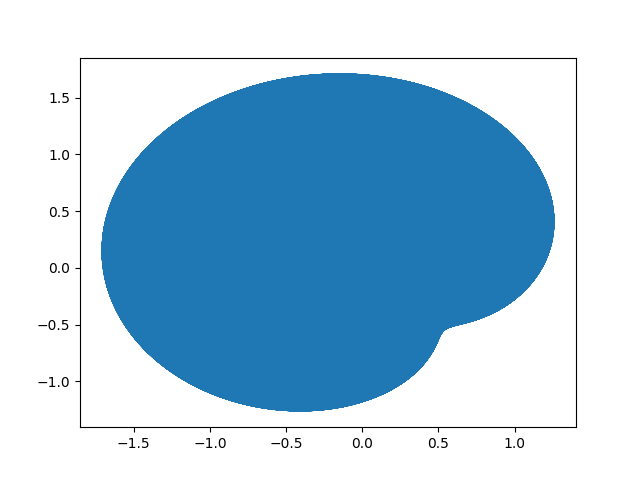}&\includegraphics[width=34mm, height=34mm]{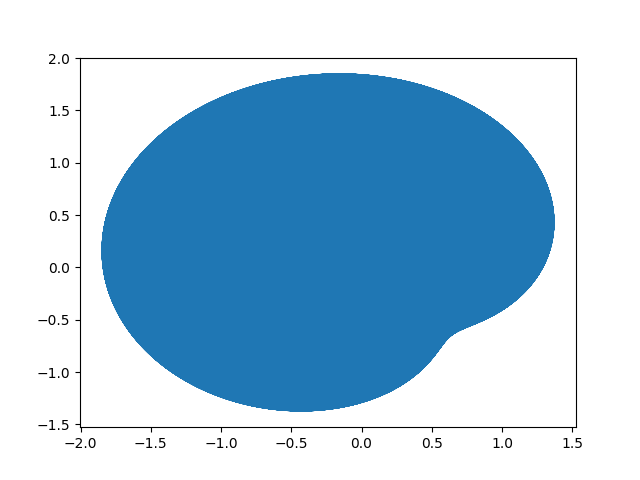}\\
$v_2=1/2$&$v_2=\sqrt{2}/2$&$v_2=\sqrt{4/5}$&$v_2=\sqrt{99/100}$\\
\end{tabular}
\end{center}

The first row of pictures shows the rotating effect of $\theta$ with fixed $v_2=v=\sqrt{1/2},\varphi=\pi/3$, note that the effect of $\varphi$ is the same as $\theta$ from the expression. The second row of pictures shows how $v_2$ will change the boundaries, with $\theta=\pi/4,\varphi=3\pi/4$. In particular, the cardioid becomes solid when $v_2\geq v.$ 

Let us return to the proof of Lemma~\ref{lem.KeyLem}. Let $\tilde{\mathcal{W}}_{A}$ be an open neighborhood of $A\in\mathcal{S}_\theta$ in $\mathrm{SU}(1,1)$ such that, assuming $$A=\frac{1}{\sqrt{1-v^2}}\begin{bmatrix}e^{i\theta}&ve^{-i\varphi}\\ve^{i\varphi}&e^{-i\theta}\end{bmatrix},$$ and letting $\theta,\varphi$  in \eqref{eq.parameterization} be the same as in $A$, for any $B_j\in\tilde{\mathcal{W}}_A, j=1,2,3$, there exists $(\varphi_1,v_2,\varphi_3)$ such that $B_1B_2B_3=\eta(A_1, A_2, A_3)=A_3A_2A_1.$

Now that $\eta$ sets up a local homeomorphism between $\mathcal{S}_\theta$ and $\mathrm{SU}(1,1)$, we can construct the required maps. Let $\mathcal{W}_{A,K}$ be the intersection of $\tilde{\mathcal{W}}_A$ with $C_{A,K}$. For $B\in\mathcal{W}_{A,K}$ , let $\Phi(B)(\omega)=A(\omega)$ if $\omega \notin \bigcup_{i=-1}^{1}T^{i}(K)$ and for $\omega \in K$, let 
$$
(\Phi(B)(T^{-1}(\omega)),\Phi(B)(\omega),\Phi(B)(T(\omega)))=\eta^{-1}(B(T^{-1}(\omega))\cdot B(\omega)\cdot B(T(\omega))).
$$
Let $\Psi(B)(\omega)=id$ for $\omega \notin K \cup T(K)$ and let $\Psi(B)(\omega)=\Phi(B)(T^{-1}(\omega))\cdot[B(T^{-1}(\omega))]^{-1}$ for $\omega\in K$. For $\omega \in T(K)$, let 
$$
\Psi(B)(\omega)=\Phi(B)(T^{-1}(\omega))\cdot\Phi(B)(T^{-2}(\omega))\cdot[B(T^{-2}(\omega))]^{-1}\cdot[B(T^{-1}(\omega))]^{-1},
$$ 
and all required properties follow.
\end{proof}

Now let us prove Lemma~\ref{lem.project}:

\begin{proof}[Proof of Lemma \ref{lem.project}]
Let $V\subset \Omega$ be a non-empty open subset such that letting $K=\overline{V}$, $K\cap T(K)=\emptyset$, and $K\cap T^2(K)=\emptyset.$ Then let $\Phi_{A,V}:\mathcal{W}_{A,V}\to C_{A,K}(\Omega,\mathrm{SU}(1,1))$ and $\Psi_{A,V}:\mathcal{W}_{A,V}\to C(\Omega,\mathrm{SU}(1,1))$ be given by Lemma~\ref{lem.ToLocal}. Let $\Phi_{A,K}:\mathcal{W}_{A,K}\to C(\Omega,\mathcal{S}_\theta)$ and $\Psi_{A,K}:\mathcal{W}_{A,K}\to\mathrm{SU}(1,1)$ given by Lemma~\ref{lem.KeyLem}. Let $\mathcal{W}$ be the domain of $\Phi_{A,K} \circ \Phi_{A,V}$, $\Phi = \Phi_{A,K}\circ \Phi_{A,V}$ and $\Psi = (\Psi_{A,K}\circ\Phi_{A,V})\cdot\Psi_{A,V}$, then the result is proved. 
\end{proof}

\begin{remark}
One may compare Lemma~\ref{lem.KeyLem} with \cite[Lemma 1.3]{ABD12} and the proof of \cite[Lemma 11]{ABD09}. Our choice of the projection is not a coincidence. Indeed, for Schr\"odinger cocycles, each contains only one parameter and therefore one needs the triple product to cover $\mathrm{SL}(2,\bbR)$. For Szeg\H{o} cocycles with constant $\theta$, each matrix contains two parameters: the argument $\varphi$ and the modulus $v$, so taking a triple product will produce an underdetermined system. One can either choose the argument or the modulus of a Szeg\H{o} cocycle as the parameter. However, we do need at least two directions to cover an open region in the complex plane. Taking three argument parameters will cause a boundary problem. Namely, the triple product of $A\in \mathcal{S}_\theta$ with fixed modulus $\lambda$ is always on the boundary of $G(\lambda)$ (the tip). Thus perturbations of the upper-left element of $A$ may generate triple products that are outside of $G(\lambda)$; as a consequence, the inverse map cannot be defined.  
\end{remark}

\section{The Projection Lemma for Jacobi Matrices}\label{JacobiProject}

Define  the set of Jacobi cocycles as follows:
\begin{equation}\label{eq.JacobiM}\mathcal{J}=\left\{\frac{1}{a}\begin{bmatrix}t&-1\\a^2&0\end{bmatrix}:a>0,t\in\bbR\right\}.\end{equation}
Then we have the following:

\begin{lemma}[Projection Lemma: Jacobi case]\label{lem.JprojectionLem}
Let $T : \Omega \to \Omega$ be a minimal homeomorphism of a compact metric space with at least four points, and let $A\in C(\Omega,\mathcal{J})$. Then there exist a neighborhood $\mathcal{W}\subset C(\Omega,\mathrm{SL}(2,\bbR))$ of $A$ and continuous maps
$$
\Phi = \Phi_A : \mathcal{W} \to C(\Omega,\mathcal{J})~\text{and}~\Psi=\Psi_A:\mathcal{W}\to C(\Omega,\mathrm{SL}(2,\bbR))
$$
such that 
$$
\Psi(B)(T(\omega))B(\omega)[\Psi(B)(\omega)]^{-1}=\Phi(B)(\omega),
$$
$$
\Phi(A)=A,~\Psi(A)=id.
$$
In particular, an $\mathrm{SL}(2,\bbR)$ valued perturbation of an element in $\mathcal{J}$ can be conjugated to a $\mathcal{J}$-valued perturbation.
\end{lemma}

By Lemma~\ref{lem.ToLocal}, we only need to deal with localized perturbations. 

\begin{lemma}\label{lem.KeyLem0}
Let $K \subset \Omega$ be a compact set such that $K\cap T(K)=\emptyset$ and $K\cap T^2(K)=\emptyset$. Let $A\in C(\Omega,\mathcal{J})$. Then there exist an open neighborhood $\mathcal{W}_{A,K}\subset C_{A,K}(\Omega,\mathrm{SL}(2,\bbR))$ of $A$ and continuous maps 
$$
\Phi = \Phi_{A,K} : \mathcal{W}_{A,K}\to C(\Omega,\mathcal{J}) \quad \text{and}\quad \Psi=\Psi_{A,K}:\mathcal{W}_{A,K}\to C(\Omega,\mathrm{SL}(2,\bbR))
$$
such that for $B\in \mathcal{W}_{A,K}$,
$$
\Psi(B)(T(\omega)) \cdot B(\omega) \cdot [\Psi(B)(\omega)]^{-1}=\Phi(B)(\omega)
$$
$$
\Phi(A)=A,~\Psi(A)=id.
$$
\end{lemma}

\begin{proof}
The proof follows the same outline as the proofs of \cite[Lemma 11]{ABD09} and \cite[Lemma 1.3]{ABD12} with minor changes. For any compact $K\subset \Omega$ with non-empty interior, if $T$ is minimal, then by \cite[Lemma 10]{ABD09}, namely, the $\mathrm{SL}(2,\bbR)$ version of Lemma~\ref{lem.ToLocal}, any $\mathrm{SL}(2,\bbR)$ perturbation $\tilde{A}$ of $A\in\mathcal{J}$ can be conjugated to a perturbation $\hat{A}$ such that $\hat{A}=A$ outside $K$. Therefore it suffices to consider these local perturbations.

The argument is divided into two different cases:

\textbf{Case 1: $\tr \, A$ vanishes identically.} 

Since $\Omega$ has at least four points, we can pick $K$ having a compact neighborhood $K'$ such that $K'\cap T^{j}(K')=\emptyset$ for $j=1,2,3.$

Since $A\in \mathcal{J}$ and $\tr \, A=0$ together give $A^{4}=id$, as $\hat{A}$ is close to $A$, $\hat{A}^{4}$ is close to the identity $id$. Define $E=\sup_{\omega\in \Omega}\Vert \hat{A}^4(\omega)-id\Vert^{\frac{1}{2}}$.

For each $\omega \in K'$, let $E_3(\omega)=E\phi(\omega)$, where $\phi(\omega)$ is determined by 
$$
\hat{A}^4(\omega)=id+\begin{bmatrix}p(\omega)&q(\omega)\\r(\omega)&s(\omega)\end{bmatrix}=A_4A_3A_2A_1,
 ~A_i(\omega)=\frac{1}{a_i(\omega)}\begin{bmatrix}E_i(\omega)&-1\\a_i^2(\omega)&0\end{bmatrix}, i=1,2,3,4
$$ 
and 
$$E_1(\omega)=-\frac{r(\omega)+E\frac{a_1a_4}{a_2a_3}\phi(\omega)}{s(\omega)+1},E_2(\omega)=\frac{a_2^2-\frac{a_1a_2a_3}{a_4}(1+s(\omega))}{E\phi(\omega)},E_4(\omega)=\frac{a_4^2q(\omega)-\frac{a_3a_4}{a_1a_2}E_2(\omega)}{s(\omega)+1},$$
with the expression for $E_2(\omega)=0$ whenever $E\phi(\omega)=0$. It is readily verified that $E_i(\omega): K'\to\bbR$ are continuous and vanish on $\partial K'$.

Construct $\Phi(\hat{A})=\hat{A}=A$ for $x\notin \cup_{j=1}^{3}T^{j}(K')$. For $\omega\in K'$, let $\Phi(\hat{A})(T^{j}(\omega))=A_{j+1}(\omega)$, $j=0,1,2,3.$ Then $\Phi$ is continuous in a neighborhood of $A$.

Since $\Phi(\hat{A})^4(\omega)=\hat{A}^4(\omega)$ for $\omega \in K'$, let $\Psi(\hat{A})=id$ for $\omega \notin \cup_{j=1}^{3} T^{j}(K')$, $\Psi(\hat{A})(T^{j}(\omega))=\Phi(\hat{A})^j(\omega)\hat{A}^{j}(\omega)^{-1}$ for $\omega\in K'$ and $j=1,2,3$. Then $\Psi$ is continuous in a neighborhood of $A$ and satisfies 
$$
\Psi(\hat{A})(T(\omega))\hat{A}(\omega)[\Psi(\hat{A})(\omega)]^{-1}=\Phi(\hat{A})(\omega).
$$ 
\medskip

\textbf{Case 2: $\tr \, A$ does not vanish identically.}

In this case, we do not need to take four copies of $\mathcal{J}$. Instead, we take three copies as in the proof of Lemma~\ref{lem.project}. This is essentially \cite[Lemma 11]{ABD12} with minor notational changes. Since $\tr \, A$ does not vanish identically, we can start with an open $V\subset X$ in which $\tr \, A \neq 0$ and $K=\overline{V}$ such that $K\cap T(K)=K\cap T^2(K)=\emptyset.$

Let $A_j=\frac{1}{a_j}\begin{bmatrix}t_j&-1\\a_j^2&0\end{bmatrix},t_j\in\bbR,j=1,2,3.$ A computation gives
$$
A_3A_2A_1=\frac{1}{a_1a_2a_3}\begin{bmatrix}t_1t_2t_3-a_2^2t_1-a_1^2t_3)&a_2^2-t_2t_3\\a_3^2t_2t_1-a_1^2a_3^2&-a_3^2t_2\end{bmatrix}=\begin{bmatrix}p&q\\r&s\end{bmatrix},
$$
where $s\neq 0,$ implies $t_1=-\frac{r+\frac{a_1a_3}{a_2}}{s}, t_2=-\frac{a_1a_2s}{a_3}, t_3=\frac{a_3^2q-\frac{a_2a_3}{a_1}}{s}$. Define $\eta:\mathcal{J}\to \mathcal{L}\subset \mathrm{SL}(2,\bbR)$, where any $\begin{bmatrix}p&q\\r&s\end{bmatrix}\in\mathcal{L}$ satisfies $s\neq 0,$ by $$\eta: \mathcal{J}\to \mathcal{L}$$
$$
(A_1,A_2,A_3)\to A_3A_2A_1.
$$
Then $\eta$ is an analytic homeomorphism and its inverse is well defined.

Let $\mathcal{W}_{A,K}$ be a neighborhood of $A$ in $C_{A,K}$. For $B\in\mathcal{W}_{A,K}$ , let $\Phi(B)(\omega)=A(\omega)$ if $\omega\notin\bigcup_{i=-1}^{1}T^{i}(K)$ and for $\omega\in K$, let 
$$
(\Phi(B)(T^{-1}(\omega)),\Phi(B)(\omega),\Phi(B)(T(\omega)))=\eta^{-1}(B(T^{-1}(\omega))\cdot B(\omega)\cdot B(T(\omega))).
$$
Let $\Psi(B)(\omega)=id$ for $\omega\notin K\cup T(K)$ and let $\Psi(B)(\omega)=\Phi(B)(T^{-1}(\omega))\cdot[B(T^{-1}(\omega))]^{-1}$ for $\omega\in K$. For $\omega\in T(K)$, let 
$$
\Psi(B)(\omega)=\Phi(B)(T^{-1}(\omega))\cdot\Phi(B)(T^{-2}(\omega))\cdot[B(T^{-2}(\omega))]^{-1}\cdot[B(T^{-1}(\omega))]^{-1}
$$ 
and all required properties follow.
\end{proof}

\begin{proof}[Proof of Lemma \ref{lem.JprojectionLem}]
Let $V\subset \Omega$ be a non-empty open subset such that letting $K=\overline{V}$, $K\cap T(K)=\emptyset$, and $K\cap  T^2(K)=\emptyset.$ Then let $\Phi_{A,V}:\mathcal{W}_{A,V}\to C_{A,K}(\Omega,\mathrm{SL}(2,\bbR))$ and $\Psi_{A,V}:\mathcal{W}_{A,V}\to C(\Omega,\mathrm{SL}(2,\bbR))$ be given by Lemma~\ref{lem.ToLocal}. Let $\Phi_{A,K}:\mathcal{W}_{A,K}\to C(\Omega,\mathcal{J})$ and $\Psi_{A,K}:\mathcal{W}_{A,K}\to\mathrm{SL}(2,\bbR)$ given by Lemma~\ref{lem.KeyLem}. Let $\mathcal{W}$ be the domain of $\Phi_{A,K}\circ\Phi_{A,V}$, $\Phi=\Phi_{A,K}\circ\Phi_{A,V}$ and $\Psi=(\Psi_{A,K}\circ\Phi_{A,V})\cdot\Psi_{A,V}$, then the result is proved. 
\end{proof}

\section{Opening Gaps for Jacobi and CMV Matrices in the Continuous Case}\label{mainProof} 

Now we are ready to give the proof of Theorem~\ref{thm.main1} and Theorem~\ref{thm.main3}. We note that we can restrict our attention to the case where $\Omega$ is not finite. Indeed, if $\Omega$ is finite, then we are dealing with the periodic case since $T$ is assumed to be minimal and hence has to be of the form $T \omega = \omega + 1$ on $\Omega = \mathbb{Z} / (p\mathbb{Z})$.  Thus, generic gap opening in the form stated in Theorems~\ref{thm.main1} and \ref{thm.main3} is already known; see Simon \cite{Simon76} (where the statement is proved for continuum Schr\"odinger operators -- the reader can verify that the same argument yields the desired statement in the case of Jacobi matrices) and Simon \cite{Simon2} (see Theorem~11.13.1 covering the CMV case in the form needed here).

Let us proceed to the proofs under the assumption justified above. For the Jacobi case:

\begin{proof}[Proof of Theorem \ref{thm.main1}]
Take a sampling function $b \in C(\Omega,\bbR)$ and a label $\ell\in\mathcal{L}_{S}\cap(0,1)$. Suppose that $\ell$ is the label of a collapsed gap of $\Sigma_{b}$, that is, there exists a unique $E_0\in\bbR$ such that 
$$
k_{b}(E_0)=\ell.
$$ 
Let 
$$
A=\frac{1}{a}\begin{bmatrix}E_0-b&-1\\a^2&0\end{bmatrix}.
$$

Since $2\rho_{b}(E_0)=1-k_{b}(E_0)=1-\ell$, compare \cite{DFZ23}, we can apply Lemma~\ref{thm.AUH} to obtain a continuous path $A_{t},t\in[0,1]$ with $A_{0}=A$ and $A_{t}\in\mathcal{UH}$ for each $t>0$. More importantly, $A_t$ preserves the label.

Let $\tilde{A}_{t}=\Phi(A_{t})$ for small enough $t$, where $\Phi$ is given by Lemma~\ref{lem.JprojectionLem} applied to $A_0$.  Since $\tilde{A}_{t}\in\mathcal{J}$, there exists a continuous path $b_{t}\in C(\Omega,\bbR)$ such that 
$$
\tilde{A}_{t}=\frac{1}{a}\begin{bmatrix}E_0-b_t&-1\\a^2&0\end{bmatrix}.
$$
For each $t>0$, we have $\tilde{A}_{t}\in\mathcal{UH}$ and therefore $E_0\notin\Sigma_{b_{t}}$. Thus $\ell$ is the label of an open gap of $\Sigma_{b_t}$ for $t>0$.
\end{proof}

For the CMV case:

\begin{proof}[Proof of Theorem \ref{thm.main3}]
Take a sampling function $v\in C(\Omega,\bbD)$ and a label $\ell\in\mathcal{L}_S\cap(0,1)$. Suppose that $\ell$ is the label of a collapsed gap of $\Sigma_{v}$, that is, there exists a unique $z_0=e^{i\theta_0}\in\partial \bbD$ such that 
$$
k_{v}(z_0)=\ell.
$$ 
Let 
$$
A=\frac{1}{\sqrt{1-|v|^2}}\begin{bmatrix}z_0^{\frac{1}{2}}&-\overline{v}z_0^{\frac{1}{2}}\\-vz_0^{\frac{1}{2}}&z_0^{-\frac{1}{2}}\end{bmatrix}.
$$

Since $2\rho_{v}(z_0)=k_{v}(z_0)=\ell$, see \cite[Chapter 8.3]{Simon2}, we can apply Lemma \ref{thm.AUH} to obtain a continuous path $A_{t},t\in[0,1]$ with $A_{0}=A$ and $A_{t}\in\mathcal{UH}$ for each $t>0.$ More importantly, $A_t$ preserves the label.

Let $\tilde{A}_{t}=\Phi(A_{t})$ for small enough $t$, where $\Phi$ is given by Lemma \ref{lem.project} applied to $A_0$. Since $\tilde{A}_{t}\in\mathcal{S}_{\theta_0}$, there exists a continuous path $v_{t}\in C(\Omega,\bbD)$ such that 
$$
\tilde{A}_{t}=\frac{1}{\sqrt{1-|v_t|^2}}\begin{bmatrix}z_0^{\frac{1}{2}}&-\overline{v_t}z_0^{-\frac{1}{2}}\\-v_tz_0^{\frac{1}{2}}&z_0^{-\frac{1}{2}}\end{bmatrix}.
$$
For each $t>0$, we have $\tilde{A}_{t}\in\mathcal{UH}$ and therefore $z_0\notin\Sigma_{v_{t}}$. Thus $\ell$ is the label of an open gap.
\end{proof}

\section{Opening Gaps for Jacobi and CMV Matrices in the Analytic Case}\label{analyticCase}

In this section, we will show that in the subcritical region (see Definition \ref{def.global}), for both Jacobi and CMV matrices, the boundaries of resonance tongues are analytic. As a corollary of analytic tongue boundaries and an argument of transversality at the tongue tips, we give a proof for generic gap opening in the subcritical region for both Jacobi and CMV matrices with analytic sampling functions and {\it Diophantine} frequencies. 

Recall that $\alpha\in\T$ is called Diophantine if there exist constants $\kappa, \tau > 0$ such that
$$
\inf \limits_{j \in\mathbb{Z}} \vert n\alpha - j \vert \geq \frac{\kappa}{\vert n\vert^{\tau}}
$$
for all $n \in \mathbb{Z} \backslash \{0\}$. Let $DC$ be the set of all Diophantine frequencies $\alpha$, it is well known that $DC$ is of full measure.

In the following definition, we consider a one parameter family of Jacobi or CMV matrices by assigning a coupling constant to the sampling function, that is $\delta b$ for Jacobi case and $\delta h$ for CMV case, where $\delta$ varies in $[0,1]$.
\begin{definition}[Resonance tongue]
Let $(\alpha,A_{\mu}(x))$ be a $\chi$-valued cocycle, where $\chi=\mathrm{SL}(2,\bbR)$ or $\mathrm{SU}(1,1)$ and $\mu=(E,\delta)$ is the multiparameter, where $E$ is the spectral parameter and $\delta$ is the coupling constant.
The resonance tongue corresponding to a label $k\in\mathbb{Z}$ is defined as
$$\mathcal{R}_{k}=\{\mu\in\bbR^2:2\rho(\alpha,A_{\mu})=k\alpha\}.$$
Its boundaries are  functions $E^{k}_{\pm}=E^{k}_{\pm}(\delta)$ of the coupling constants. We say $\mathcal{R}_{k}$ has analytic tongue boundaries if $E^{k}_{\pm}(\delta)$ are real analytic functions of $\delta.$
\end{definition}

\begin{definition}[Cocycle classification]\label{def.global} Let $T:\T\to\T$ be the irrational rotation $Tx=x+\alpha,\alpha\in\bbR\setminus\bbQ$,
and $(\alpha,A)$ be an analytic $\chi$-valued cocycle that is not uniformly hyperbolic. Then  $(\alpha,A)$ is said to be
\begin{enumerate}[itemsep=1ex]
    \item \emph{Supercritical}, if $\sup_{z}\Vert A^n(z)\Vert$ grows exponentially.
    \item \emph{Subcritical}, if there exists a uniform sub-exponential upper bound on the growth of $\Vert A^n(z)\Vert$ through some band $|\Im z|<\epsilon.$ 
    \item \emph{Critical}, otherwise.
\end{enumerate}
\end{definition}

Let's first introduce the following assumptions and a criterion of \cite{LDZ22}:

Let $\mu_0=(E_0,\delta_0)$ be any fixed parameter such that the cocycle $(\alpha,A_{\mu_0})$ is subcritical and let $\mu$ vary in its neighborhood
\begin{itemize}
\item[(H1):]  $\rho(\alpha,A_{\mu})$ satisfies $2 \rho(\alpha,A_{\mu_{0}})=\langle k,\alpha\rangle\mod\mathbb{Z}$ and is a non-decreasing (or non-increasing) function of $E$ in a neighborhood of $E_{0}$.
\item[(H2):]  $(\alpha,A_{\mu_{0}})$ is not uniformly hyperbolic.
\item[(H3):]  There exists a neighborhood of $\mu_0$ such that $(\alpha,A_{\mu})$ is subcritical.
\end{itemize}
We further assume  that the cocycle $(\alpha,A_{\mu_0})$ is reducible, that is, there exists $B\in C^{\omega}(\T,\mathrm{SL}(2,\bbR))$ such that 
$$B(\omega+\alpha)A_{\mu_0}(\omega)[B(\omega)]^{-1}=\begin{bmatrix}1&c\\0&1\end{bmatrix}.$$
Denote $\mu-\mu_0=(\Delta,\delta)$ and let $P(\omega)$ be the collection of the linear terms of $(\Delta, \delta)$, that is,
$$B(\omega+\alpha)A_{\mu}(\omega)[B(\omega)]^{-1}=\begin{bmatrix}1&c\\0&1\end{bmatrix}(id+P(\omega)+O_2(\Delta,\delta)),$$
where $O_{2}$ is the collection of higher order terms. In the Jacobi case, the $O_2$ term is actually zero. In the following context, we use $[\cdot]$ for the average over $\T$ in Lebesgue measure.
\begin{lemma}\label{twoCases}
Let $\alpha \in \mathrm{DC}, \mu_{0}=(E_{0},\delta_{0})$ be fixed and assume that for a suitable neighborhood of $\mu_{0}$, the hypotheses {\rm (H1), (H2)} and {\rm (H3)} are satisfied. Then we have the following:
\begin{enumerate}
\item If $c\neq 0$ and  the coefficient of $\Delta$ in $[P_{21}]$ is strictly non-zero,
then the tongue boundaries $E_{\pm}=E_{\pm}(\delta)$ are real analytic in a neighborhood of $\delta_{0}$.
\item  If $c=0$ and  the coefficient of $\Delta^{2}$ in the expression of $\det [P]$ is non-zero, then the two tongue boundaries $E_{\pm}=E_{\pm}(\delta),$ with $E_{\pm}(\delta_{0})=E_{0}$ are real analytic in a neighborhood of $\delta_{0}$.
\end{enumerate}
\end{lemma}

One can consult \cite{LDZ22} for the detailed proof. We also need the following global to local reduction of \cite{LYZZ}:
\begin{lemma}\label{ARC}
Let $\alpha\in DC$, assume that $(\alpha, A_{E,\delta})$ is subcritical, then there exist $Z_{E}\in C^{\omega}(2\T,\mathrm{SL}(2,\bbR))$, $f_{E}(\theta)\in C^{\omega}(\T,\mathrm{sl}(2,\bbR))$, $\phi(E)\in\bbR$, $D_{\phi(E)}=\mathrm{diag}\{e^{i\phi(E)},e^{-i\phi(E)}\}$ such that
\begin{equation}\label{eq.conjugate}
Z^{-1}_{E}(\omega+\alpha)A_{E,\delta}(\theta)Z_{E}(\omega)=D_{\phi(E)}e^{f_{E}(\theta)}
\end{equation}
with $\|f_{E}\|_{h}<\eta,\|Z_{E}\|_{h}<\Gamma(\alpha,\eta,\lambda)$ for some constant $\Gamma$.
\end{lemma}
The point of Lemma \ref{ARC} is that a subcritical cocycle can be conjugate to a near constant one. It is straightforward to parse this result to an $\mathrm{SU}(1,1)$ version. 
\subsection{Jacobi Case}
Let $a(\omega)>0$ be fixed and $b(\omega)\in C^{\omega}_{r}(\T,\bbR)$ for some $r>0$. In the following context, we will consider $b_{\delta}(\omega)=\delta b(\omega)$ and denote $(\alpha,A_{E,\delta}(\omega))$
the corresponding Jacobi cocycle. 
\begin{theorem}[Analytic tongue boundaries]\label{thm.JacobiBound}
Let $\alpha\in DC$ and
assume that $(\alpha,A_{\mu_0})$ is subcritical with $2\rho(\alpha,A_{\mu_0})=k\alpha$, then the boundaries of $\mathcal{R}_{k}$ are analytic functions of $\delta$ in the neighborhood of  $(E_0,\delta_0)$ where the cocycle $(\alpha,A_\mu)$ is subcritical. 
\end{theorem}
\begin{proof}By Lemma \ref{ARC}, there exists $Z_1\in C^{\omega}(2\T,\mathrm{SL}(2,\bbR))$
such that $$Z_1(\omega+\alpha)A_{E_0,\delta_0}(\omega)[Z_1(\omega)]^{-1}=A'(\omega),$$
where $A'\in C^{\omega}(\T,\mathrm{SL}(2,\bbR))$ is close to a constant. For $(\alpha,A')$, by Eliasson's KAM scheme \cite{Eli92, HA09} and that $2\rho(\alpha,A_{\mu_0})=k\alpha$, there exists $Z_2\in C^{\omega}(2\T,\mathrm{SL}(2,\bbR))$ such that $$Z_2(\omega+\alpha)A'(\omega)[Z_2(\omega)]^{-1}=\begin{bmatrix}1&c\\0&1\end{bmatrix}.$$
Let $B=Z_2Z_1\in C^{\omega}(2\T,\mathrm{SL}(2,\bbR))$, then it follows that  
\begin{equation}\label{eq.conj1}B(\omega+\alpha)A_{E_0, \delta_0}(\omega)[B(\omega)]^{-1}=\begin{bmatrix}1&c\\0&1\end{bmatrix}.
\end{equation}
Perturb $(E_0, \delta_0)$ by $(E_0+\Delta,\delta_0+\delta )$, where $\Delta,\delta\in\bbR$ are small.
Then the same $B$ conjugates the cocycle $(\alpha,A_{E_0+\Delta,\delta_0+\delta })$ to \begin{equation}\label{eq.conj2}B(\omega+\alpha)A_{E_0+\Delta,\delta_0+\delta}(\omega)[B(\omega)]^{-1}=\begin{bmatrix}1&c\\0&1
\end{bmatrix}(id+ P(\omega,\mu)).
\end{equation} In order to apply Lemma \ref{twoCases}, we need to compute the exact expression of $P$.

It follows from \eqref{eq.conj1}, \eqref{eq.conj2}, and a straightforward computation that for $P(\omega)=\begin{bmatrix}P_{11}&P_{12}\\P_{21}&-P_{11}\end{bmatrix}$, we have
\begin{equation}
\begin{aligned}
&P_{11}=-(\Delta-\delta b)(b_{12}b_{22})\\
&P_{12}=(\Delta-\delta b)b_{12}^2\\
&P_{21}=-(\Delta-\delta b)b_{22}^2.
\end{aligned}
\end{equation}
In order to apply Lemma \ref{twoCases}, if $c\neq 0$,  we need to verify that $[b_{22}^2]\neq 0$. Since $B$ is analytic, if $[b_{22}^2]=0,$ then $b_{22}\equiv 0$, which leads to $b_{12}\equiv 0$ by \eqref{eq.conj1}. This contradiction implies that $[b_{22}^2]>0$.
If $c=0,$  we need to verify that $$-[b_{12}^2][b_{22}^2]+[b_{12}b_{22}]^2\neq 0.$$

By the Cauchy-Schwarz inequality, $-[b_{12}^2][b_{22}^2]+[b_{12}b_{22}]^2\leq 0$,  equality happens if and only if $b_{12}(\omega)=b_{22}(\omega)$ for all $\omega$. This is again impossible due to \eqref{eq.conj1}. Indeed, assuming that 
$$B(\omega+\alpha)A_{E_0,\delta_0}(\omega)[B(\omega)]^{-1}=id,$$
we have $b_{11}(\omega+\alpha)=-a(\omega+\alpha)b_{12}(\omega)=-a(\omega)b_{22}(\omega)=b_{21}(\omega+\alpha)$, by continuity and minimality, $b_{11}(\omega)=b_{12}(\omega)$ which contradicts the linear independence. Hypotheses (H1)-(H3) are trivial for $(\alpha,A_{\mu})$. By Lemma \ref{twoCases}, the tongue boundaries of $\mathcal{R}_k$ are analytic for all $k\in\mathbb{Z}$ in the subcritical region.
\end{proof}
\medskip

To show that the boundaries of $\mathcal{R}_{k}$ do not coincide all the time, we need a transversality argument for the case $c=0$ in each connected component of subcritical region. We noticed that the computations become much easier in $\mathrm{SU}(1,1).$ We need the following result of \cite{LDZ22}:
\begin{lemma}\label{lem.nonzero}
If $\det[P]>0$, then the rotation number
$\rho(\alpha,(id+P(\cdot))$
 is strictly non-zero.
\end{lemma}

\begin{theorem}[Transversality]\label{thm.transver}
Assume that $(E_0,\delta_0)$ is such that $(\alpha,A_{E_0,\delta_0})$ is subcritical and $E_0=\mathcal{R}_{k}\cap\{(E,\delta_0)\}$.
There exists an open and dense subset of $b\in C^{\omega}(\T,\bbR)$ for which $\frac{d E^{+}_{k}}{d\delta}(\delta_0)\neq\frac{dE^{-}_{k}}{d\delta}(\delta_0).$
\end{theorem}

\begin{proof} 
Let $E_0,\delta_0$ be fixed and $M=\frac{1}{1+i}\begin{bmatrix}1&-i\\1&i\end{bmatrix}$, without loss of generality (the label $k$ might be changed), by \cite[Lemma 2.1]{KXZ20}, Lemma \ref{ARC} and Eliasson's KAM scheme, there exists $Z=\begin{bmatrix}z_1&z_2\\\overline{z_2}&\overline{z_1}\end{bmatrix}\in C^\omega(2\T,\mathrm{SU}(1,1))$ 
such that \begin{equation*}
Z(\omega+\alpha)MA_{E_0,\delta}(\omega)M^{-1}[Z(\omega)]^{-1}=\begin{bmatrix}e^{ik\alpha/2}&0\\0&e^{-ik\alpha/2}\end{bmatrix}.
\end{equation*}
Let $B=\begin{bmatrix}e^{-ik\omega/2}&0\\0&e^{ik\omega/2}\end{bmatrix}Z$, then
\begin{equation}\label{eq.conjugation1}B(\omega+\alpha)MA_{E_0,\delta_0}(\omega)M^{-1}[B(\omega)]^{-1}=id.\end{equation}
Let $\mu=(\Delta, \delta)$ be small, and let $P\in\mathrm{su}(1,1)$ be such that 
$$B(\omega+\alpha)MA_{E_0+\Delta,\delta_0+\delta}M^{-1}B(\omega)=id+P(\omega,\mu),$$
then $P(\cdot,\mu)$ can be computed explicitly as the following:
\begin{equation*}
\begin{aligned}
&P_{11}=i\frac{\Delta-\delta b}{2}(|b_{1}|^2+|b_{2}|^2-\overline{b_1}b_2-b_1\overline{b_2}),\\
&P_{12}=i\frac{\Delta-\delta b}{2}(b_1-b_2)^2,
\end{aligned}
\end{equation*}
where $b_1=z_1e^{-ik\omega/2},b_2=z_2e^{-ik\omega/2}$. By Lemma \ref{lem.nonzero}, the boundaries $(\Delta(\delta),\delta)$ are solutions to the quadratic equation $\det[P](\Delta,\delta)=0.$ Equivalently,
\begin{equation}\label{eq.detPzero}
(\hat{h}_0\Delta-\hat{(bh)}_{0}\delta)^2-|\hat{g}_{k}\Delta-\hat{(bg)}_{k}\delta|^2=0,
\end{equation}
where $h=|z_1|^2+|z_2|^2-\overline{z_1}z_2-z_1\overline{z_2}, g=(z_1-z_2)^2.$ The simple observation that tongue boundaries always exist implies that  the discriminant of \eqref{eq.detPzero} must be nonnegative, that is,
\begin{equation}\label{eq.criticalCondition}
(\Re(\hat{g}_k)\overline{\hat{(bg)}_{k}})^2-|\hat{g}_k\hat{(bg)}_{k}|^2+\hat{(bh)}_0^2|\hat{g}_{k}|^2+\hat{h}_0^2|\hat{(bg)}_{k}|^2-2\Re(\hat{g}_k\overline{\hat{(bg)}_{k}})\hat{h}_0\hat{(bh)}_0\geq 0.
\end{equation}
Let $S_k=\{b\in C^\omega(\T,\bbR):\text{the LHS of }\eqref{eq.criticalCondition} \text{ is nonzero}\}$, then clearly $S_k$ is open. We need to show that $S_k$ is also dense. Let $f(\omega)\in C^{\omega}(\T,\bbR)$ be arbitrarily small and assume that $b$ is such that \eqref{eq.criticalCondition} equals zero. By \eqref{eq.conjugation1}, we have 
\begin{equation}\label{eq.conjugation2}
B(\omega+\alpha)M\frac{1}{a(\omega)}\begin{bmatrix}E_0-\delta_0 b(\omega)-\delta_0 f(\omega)&-1\\a^2(\omega)&0\end{bmatrix}M^{-1}[B(\omega)]^{-1}=id+Q(\omega),
\end{equation}
where \begin{equation}\label{eq.newPerturb}
Q(\omega)=\frac{\delta_0f}{2}\begin{bmatrix}
-ih&-ige^{-ik\omega}\\
i\overline{g}e^{ik\omega}&ih
\end{bmatrix}.
\end{equation}
By a standard KAM scheme (see \cite{HA09},\cite{LYZZ}), there exists $Y\in C^{\omega}(\T,\mathrm{su}(1,1))$ such that $e^{Y}$ conjugates $id+Q(\omega)$ to a constant. Moreover, in the first inductive step, the linear part of $Y$ in $Q$ is determined by solving cohomological equations. This implies that the linear part of $Y$ has Fourier coefficients as linear combinations of those of $Q$, and therefore of $f$. As $\Vert f\Vert_r$ is sufficiently small for some $r>0$, the linear part of $Y$ dominates. Since $f$ is independent of $b$ and $Z$, there exists $f$ such that the LHS of \eqref{eq.criticalCondition} is nonzero. This proves the denseness of $S_k.$ 
\end{proof}

%\textbf{Case $\delta_0>0:$} Since $(\alpha,A_{E_0,\delta_0})$ is subcritical and $\alpha\in DC$, there exists $\tilde{B}\in C^{\omega}(\T,\mathrm{SU}(1,1))$ such that 
%\begin{equation}\label{eq.case2Conj}
%\tilde{B}(\omega+\alpha)A_{E_0,\delta_0}(\omega)[B(\omega)]^{-1}=A\in\mathrm{SU}(1,1)
%\end{equation}
%Let $U\in \mathrm{SU}(1,1)$ (\cite{KXZ20}) be such that %$$UAU^{-1}=\begin{bmatrix}e^{ik\alpha/2}&0\\0&e^{-ik\alpha/2}\end{bmatrix}.$$
%\end{proof}
%%Let $B=\begin{bmatrix}e^{-i\langle k,x\rangle/2}&0\\0&e^{i\langle %k,x\rangle/2}\end{bmatrix}U\tilde{B}$, then we have %\begin{equation}\label{eq.case2Conj2}
%B(x+\alpha)A_{E_0,\delta_0}[B(x)]^{-1}=id.
%\end{equation}
%The $k$ in \eqref{eq.case2Conj} and \eqref{eq.case2Conj2} may be different from the original label, but it has no influence on the following arguments. For simplicity, we assume that $a_{\delta_0}$ is fixed for any given $a(x)>0$.

We are ready to prove Theorem \ref{thm.main7}

\begin{proof}[Proof of Theorem \ref{thm.main7}]
For any $k$, let $\mathcal{R}_k$ be the $k$-th resonance tongue. By analyticity of $\mathcal{R}_k$'s boundaries (Theorem \ref{thm.JacobiBound}),  there are only finitely many values of $\delta$ such that boundaries coincide at these points for any fixed $(a,b)$. Excluding up to countably many values of $\delta$, no boundaries coincide. For any fixed $\delta_0$, by Theorem \ref{thm.transver}, the set of $b\in C^{\omega}(\T,\bbR)$ such that all collapsed gaps open up linearly is generic. Putting these parameters exclusion together proves the theorem.
\end{proof}

\begin{remark}
A transversality argument in each subcritical component is necessary, since it is unknown if the tongue boundaries are analytic throughout the whole parameter space. It is unclear if there is a transition to a supercritical component between two consecutive subcritical components. 
\end{remark}

\subsection{CMV case:} To be consistent with notations, let $z=e^{iE}$ be the spectral parameter of the Szeg\H{o} cocycle and $v=\lambda e^{i\delta h}$ be the sampling function in the following context. The following result of \cite{LDZ22} shows analyticity of tongue boundaries for CMV matrices in the perturbative region.
\begin{theorem}\label{thm.CMVBoundary}
Let $\lambda\in(0,1),\delta\in\bbR,r>0$,  $\alpha \in DC$, $h\in C_{r}^{\omega}(\T^{d},\bbR)$. Consider the two-sided CMV matrix with quasi-periodic Verblunsky coefficients generated by the sampling function
$$
v(\omega)=\lambda e^{i \delta h(\omega)},
$$
and denote its spectrum by  $\Sigma_{v}$.
Then there exists $\epsilon_{1} = \epsilon_{1}(\alpha,\|h\|_{r},\lambda) > 0$ such that if $|\delta_{0}|<\epsilon_{1}$ and the pair $(E_{0},\delta_{0})$ lies on a tongue boundary, we have the following:

(i) If $E_{0}$ is an endpoint of an open spectral gap of $\Sigma_{v}$, then the tongue boundary $E=E(\delta)$ such that $E_{0}=E(\delta_{0})$ is real analytic in a neighborhood of $\delta_{0}$.

(ii) If $E_{0}$ lies at a collapsed spectral gap of $\Sigma_v$, then the two tongue boundaries $E_{\pm}=E_{\pm}(\delta)$ with $E_{\pm}(\delta_{0})=E_{0}$ are real analytic in a neighborhood of $\delta_{0}$.
\end{theorem}
Then the proof of analyticity in the subcritical region follows directly from a global to local reduction Lemma \ref{ARC} and Theorem \ref{thm.CMVBoundary}.

It remains to show the transversality of tongue boundaries where a gap collapses. The following result of \cite{LDZ22} gives the proof at $\delta_0=0.$ 

\begin{theorem}\label{thm.CMVTrans}
Denote the tongue boundaries of the gap with label $k\neq 0$ by $E_{k}^{\pm}(\delta)$. Then we have $\frac{dE^{+}_{k}}{d\delta}(0)\neq \frac{dE^{-}_{k}}{d\delta}(0)$ if and only if
$\hat{h}_{k}\neq 0$.
\end{theorem}
Following the same argument as in the proof of Theorem \ref{thm.transver}, the above conclusion can be generalized to the following:

\begin{coro}\label{cor.CMVTrans}
Assume that $(E_0,\delta_0)$ is such that $(\alpha,A_{E_0,\delta_0})$ is subcritical and $E_0=\mathcal{R}_k\cap \{(E,\delta_0)\}$.
There exists an open and dense subset of $h\in C^{\omega}(\T,\bbR)$ for which  $\frac{dE^{+}_{k}}{d\delta}(\delta_0)\neq \frac{dE^{-}_{k}}{d\delta}(\delta_0)$.
\end{coro}

\begin{proof}[Proof of Theorem \ref{thm.main8}] 
For any $k$, let $\mathcal{R}_k$ be the $k$-th resonance tongue. By the analyticity of $\mathcal{R}_k$'s boundaries,  there are only finitely many values of $\delta$ such that boundaries coincide at these points for any fixed $h$. Excluding up to countably many values of $\delta$, no boundaries coincide. For any fixed $\delta_0$, by Corollary \ref{cor.CMVTrans}, the set of $h\in C^{\omega}(\T,\bbR)$ such that all collapsed gaps open up linearly is generic. Putting these parameters exclusion together proves the theorem.
\end{proof}

\section*{Acknowledgement}

We are grateful to Barry Simon for helpful conversations.

\end{document}